\documentclass{amsart}
 \usepackage{amsthm,amsmath,amssymb,hyperref}
 \usepackage[frame,cmtip,curve,arrow,matrix,line,graph]{xy}
 \usepackage{graphicx,color,float}

 \newtheorem{theorem}{Theorem}[section]
 \newtheorem{corollary}[theorem]{Corollary}
 \newtheorem{lemma}[theorem]{Lemma}
 \newtheorem{proposition}[theorem]{Proposition}
 \theoremstyle{definition}
 \newtheorem{definition}[theorem]{Definition}
 
 \theoremstyle{remark}
 \newtheorem{remark}[theorem]{Remark}
 \numberwithin{equation}{section}

 \newcommand{\Ext}{\operatorname{Ext}}
 \newcommand{\FS}{\operatorname{FS}}

 \newcommand{\mo}{\operatorname{mod}}

 \newcommand{\Hom}{\operatorname{Hom}}

  \newcommand{\Coh}{\operatorname{Coh}}

  \newcommand{\ch}{\operatorname{ch}}
  \newcommand{\Todd}{\operatorname{Todd}}
  \newcommand{\Arg}{\operatorname{Arg}}

  \newcommand{\Db}{{\mathrm D}^{b}}

  \newcommand{\bfH}{\mathbf{H}}
  \newcommand{\bfE}{\mathbf{E}}

  \newcommand{\cA}{{\mathcal A}}

  \newcommand{\cK}{{\mathcal K}}

  \newcommand{\cO}{{\mathcal O}}
  \newcommand{\cP}{{\mathcal P}}

  \newcommand{\cT}{{\mathcal T}}

  \newcommand{\bC}{{\mathbb C}}

  \newcommand{\bP}{{\mathbb P}}
  
  \newcommand{\bR}{{\mathbb R}}

  \newcommand{\bZ}{{\mathbb Z}}

\begin{document}

  \title[Quintic periods and stability conditions via HMS]{Quintic
  periods and stability conditions via homological mirror symmetry}

  \author{So Okada}

  \address{Oyama National College of Technology, Oyama Tochigi, Japan
  323-0806}

  \email{okada@oyama-ct.ac.jp}

  \thanks{This project was in part supported by JSPS Grant-in-Aid
  \#23740012.}

  \subjclass[2010]{Primary 53D37,14L24,32G20,11F11}

  \keywords{Mirror symmetry, periods, stability conditions, modular
  forms}

  \begin{abstract} 
   For the Fermat Calabi-Yau threefold and the theory of stability
   conditions \cite{Bri07}, there have been two mathematical aims given
   by physical reasoning.  One is that we should define stability
   conditions by central charges of quintic periods
   \cite{Hos04,Kon12,KonSoi13}, which extend the Gamma class \cite{KKP,
   Iri09,Iri11}.  The other is that for well-motivated stability
   conditions on a derived Fukaya-type category, each stable object
   should be a Lagrangian \cite{ThoYau}.

   We answer affirmatively to these aims with the simplest homological
   mirror symmetry (HMS for short) of the Fermat Calabi-Yau threefold
   \cite{Oka09,FutUed} and stability conditions of Bridgeland type,
   which we introduce. With HMS, we naturally obtain stability
   conditions of Bridgeland type by the monodromy around the Gepner
   point.

   As consequences, we obtain bases of quintic periods and the mirror
   map \cite{CdGP} categorically, wall-crossings by quintic periods, and
   a quasimodular form \cite{KanZag} attached to quintic periods by
   motivic Donaldson-Thomas invariants \cite{KonSoi08}.  The
   quasimodular form is of the quantum dilogarithm and of a mock modular
   form \cite{Zag06}.
  \end{abstract}
  \maketitle

  \section{Introduction}
  In this article, the Fermat Calabi-Yau threefold defined by
  $0=x_{1}^5+x_{2}^{5}+\cdots +x_{5}^{5}$ in $\bP^4$ over the complex
  number $\bC$ is said to be the quintic and denoted by $X$. The quintic
  is the central manifold in the seminal paper of the mirror symmetry
  \cite{CdGP}. Let $G\cong \bZ_{5}^{5}/\bZ_{5}$ act on coordinates of
  $\bP^{4}$ as multiplications by $\xi=\exp(\frac{2\pi i}{5})$.

     For the following {\it Picard-Fuchs equation}:
       \begin{equation*}\label{eq:PF}
	(x \frac{d}{dx})^4 - 5^5 x(x\frac{d}{dx} +\frac{4}{5})
	 (x\frac{d}{dx} +\frac{3}{5} )(x\frac{d}{dx} +\frac{2}{5})(x
	 \frac{d}{dx} +\frac{1}{5}) = 0, 
       \end{equation*} 
       we work on its solutions, which we call {\it quintic periods} and
       are not of the quintic but of its mirror family as periods of
       holomorphic 3-forms.  We look at regular singular points of the
       Picard-Fuchs equation called the large complex structure limit
       and the Gepner point (the orbifold point) corresponding to $x=0$
       and $x=\infty$.
       
       In the following, we explain that with an application in the
       theory of modular forms and periods, HMS gives certain
       categorification of quintic periods by stable Lagrangians and
       wall-crossings of stability conditions of Bridgeland type.
       
       For periods of Picard-Fuchs equations and an introduction to the
       mirror symmetry, the reader can consult \cite{Mor} (cf.
       \cite[Section 2]{KonZag}).

       \subsection{Backgrounds}\label{sec:motivations}

       HMS was introduced by Kontsevich \cite{Kon95} to give a
       categorical understanding of the mirror symmetry. HMS asserts
       derived equivalences of Fukaya-type categories and categories of
       coherent sheaves for two models in topological string theory
       \cite{Wit}.  HMS is an expanding subject \cite{BDFKK} and is
       considered as a natural framework to work on for {\it all types
       of varieties} \cite{Orl11}.  However, the original motivation on
       quintic periods themselves has not been fully pursued, partly
       because several numerical predictions of \cite{CdGP,BCOV} have
       been proved \cite{Giv,LLY,Zin} with sophisticated methods on
       equations.

       The notion of stability conditions \cite{Bri07} is categorical.
       It is based on Mumford's stabilities and Douglas'
       $\Pi$-stabilities in topological string theory
       \cite{Dou01,Dou02}.  We have expected that for a derived
       Fukaya-type category and well-motivated stability conditions,
       each object is uniquely decomposed into certain minimal
       Lagrangians as stable objects (cf. \cite[the table in
       p3]{DHKK}). This is due to an original motivation of the notion
       of stability conditions \cite{ThoYau}. This can be readily
       achieved for certain derived Fukaya-Seidel categories
       \cite{Sei00,Sei01,Sei08}, say, for ADE singularities.  However,
       we would like to use central charges of quintic periods, since
       this is based on the mirror symmetry.
       
       As for the application, there have been a number of attempts to
       attach a quasimodular or modular form to quintic periods
       \cite{Mov}, due to its intrinsic difficulty \cite{Zag12}. We
       attach a quasimodular form to quintic periods by stability
       conditions of Bridgeland type and motivic Donaldson-Thomas
       invariants \cite{KonSoi08}, using the monodromy around the Gepner
       point. In mathematics and physics, it is natural to seek a
       modular property of the generating function of geometric
       invariants of semistable objects.

	\subsection{Stability conditions of Bridgeland type}
	We use {\it central charges}, which are linear functions from
	the Grothendieck group of a triangulated category to $\bC$.  In
	Section \ref{def:stab}, we slightly relax the notion of
	stability conditions for central charges of quintic periods,
 	which we recall in Equation \ref{central}, on the heart of
	bounded a $t$-structure and introduce the notion of stability 
	conditions of Bridgeland type.

	Let us recall that a stability condition of \cite{Bri07} refines
	a heart of a bounded $t$-structure, since, up to isomorphisms,
	each non-zero object of the heart is uniquely decomposed into
	semistable objects, indexed by real numbers called {\it phases}.
	We have Jordan-H\"older decompositions of semistable objects of
	a phase by stable objects of the phase, by assuming the
	local-finiteness in {\it loc cite}. The local-finiteness easily
	holds for stability conditions of Bridgeland type discussed in
	this article.
	
	In the following, for simplicity, we call hearts of bounded
	$t$-structures as hearts and stability conditions of Bridgeland
	type, which also refine hearts, as stability conditions. To
	specify stability conditions of {\it loc cite}, we call them
	Bridgeland stability conditions.

	\subsection{Our claims:}\label{sec:claims}
	\begin{itemize}
	 \item By stability conditions of central charges of quintic
	       periods in Theorems \ref{thm:sheaf} and \ref{thm:lag},
	       HMS gives a categorical understanding of the mirror symmetry; 
	 \item To prove Theorems \ref{thm:sheaf} and \ref{thm:lag}, we
	       simply compute quintic periods asymptotically by the
	       monodromy around the Gepner point.  This is distinct from
	       previous proofs of numerical predictions of the mirror
	       symmetry;
	 \item For the original motivation of HMS, we prove any of our
	       statements without the mirror symmetry in the sense of
	       correspondences between K\"ahler and complex moduli
	       spaces.
	\end{itemize}

	For the quasimodular form in Theorem \ref{thm:quasimodular}, we
	put discussions in Sections \ref{sec:quasi_intro} and
	\ref{sec:broken_symmetry}.  Before explaining other consequences
	of Theorems \ref{thm:sheaf} and \ref{thm:lag}, let us briefly
	recall the mirror symmetry as follows.
	
	In the mirror symmetry \cite{CdGP,Giv,LLY}, we work to explain
	algebro-geometric properties of the quintic by quintic periods.
	More precisely, the most well-known equation in the mirror
	symmetry \cite{CdGP}, which predicts to give numbers of rational
	curves on a family of the quintic by quintic periods near the
	large complex structure limit, turns out wrong even if Clemens
	conjecture is true as observed by Pandharipande \cite{CoxKat}.
	However, by Gromov-Witten invariants (``virtual'' numbers of
	rational curves), for which we have an axiomatic formulation
	\cite{KonMan}, the equation has been justified \cite{Giv,LLY} by
	sophisticated methods. The famous generalization of the
	numerical prediction has been obtained in \cite{BCOV,Zin}.
	
	We have that quintic periods indeed explain non-trivial
	algebro-geometric properties of the quintic, as we have
	wall-crossings of stability conditions of the quintic given by
	quintic periods.  In Corollaries \ref{cor:analytic} and
	\ref{cor:analytic_deform}, an analytic continuation of quintic
	periods and quotients of quintic periods give {\it
	wall-crossings of second kind} \cite{KonSoi08} as tiltings
	\cite{HRS} of the heart given by the Koszul Ext algebra of
	algebro-geometric stable objects of the quintic.

	Above corollaries can be seen as resulting from certain
	categorification of quintic periods. In particular, in Corollary
	\ref{cor:PF}, we have bases of quintic periods by central
	charges of algebro-geometric stable objects, which are
	isomorphic to Lagrangian vanishing cycles.  Let us recall that
	for numerical predictions of \cite{CdGP,BCOV,Giv,LLY,Zin}, the
	mirror map, which is a quotient of quintic periods as in
	Equation \ref{eq:mirror_map}, is of the utmost importance for
	investigating the mirror symmetry.  By perturbing stability
	conditions of Theorem \ref{thm:sheaf} into ones of Lemma
	\ref{lem:sheaf_modified}, we obtain the mirror map in Corollary
	\ref{cor:mirror_map}.  We put Remark \ref{rmk:numerical} on
	numerical predictions of {\it loc cite}.
	
    	\subsection{The role of HMS in this article}
  	What HMS in the simplest form gives are algebro-geometric
  	objects or Lagrangians, by which we construct a derived
  	equivalence via the heart of the extension-closed full
  	subcategory of the objects.  On such a heart, we prove that
  	central charges of quintic periods near the large complex
  	structure limit and the Gepner point give stability conditions
  	in Theorems \ref{thm:sheaf} and \ref{thm:lag} as per the aims
	in the abstract.

	  \subsection{Theorems and corollaries}
	  Let $F$ be the function $x_{1}^{5}+\cdots+x_{5}^{5}:\bC^{5}\to
	  \bC$ and $\FS(F)$ be the derived Fukaya-Seidel category of $F$
	  defined by Lagrangian vanishing cycles in the zero locus of a
	  morsification of $F$.

	  Let us recall the following famous hypergeometric series:
	  \begin{equation*}\label{eq:hyper}
	   \omega(x,p)=\sum_{n\geq
	    0}\frac{\Gamma(1+5(n+p))}{\Gamma(1+(n+p))^{5}}x^{n+p}.
	  \end{equation*}
	  For each object $E\in \Db(\Coh X)$, the nilpotent element $J$ of
	  the second cohomology class of the quintic, and $[1:x]\in
	  \bP^{1}$, central charges of quintic periods $Z_{x}(E)$
	  \cite{Hos00,Hos04} are defined as follows:
	  \begin{equation}\label{central}
	   Z_{x}(E)=\int_{X}\ch(F)w(x,\frac{J}{2\pi i})\Todd X.
	  \end{equation}
	  We obtain quintic periods as central charges of objects.  For
	  each object $E\in \Db_{G}(\Coh X)$, we define $Z_{x}(E)$ with
	  Equation \ref{central} by forgetting the equivariance.  As
	  explained in \cite[Section 8.2.2]{HorRom},
	  $\frac{\Gamma(1+5\frac{J}{2\pi i})}{\Gamma(1+\frac{J}{2\pi
	  i})^{5}}\Todd X$ of $w(x,\frac{J}{2\pi i})\Todd X$ is the {\it
	  Gamma class} of $X$ \cite{KKP,Iri09,Iri11}.

	  The following two theorems are almost identical, but we put
	  them separately for our later discussions.  By $\mo
	  A_{3}^{\otimes 5}$, we denote the category of representations
	  of the quiver $A_{3}^{\otimes 5}$, which is recalled in
	  Section \ref{sec:HMS}.  Near the large complex structure
	  limit, we have the following theorem.

	  \begin{theorem}\label{thm:sheaf}
	   For $\Db_{G}(\Coh X)\cong \Db(\mo A_{3}^{\otimes 5})\cong
	   \FS(F)$, the heart $\mo A_{3}^{\otimes 5}$, and central charges
	   $Z_{x}$ of quintic periods near the large complex structure
	   limit $x=0$, we have stability conditions on the heart such
	   that each stable object is isomorphic to a Lagrangian vanishing
	   cycle and an equivariant coherent sheaf of the Beilinson basis
	   with a shift.
	  \end{theorem}

	  Near the Gepner point, we have the following theorem.

	  \begin{theorem}\label{thm:lag}
	   For $\Db_{G}(\Coh X)\cong \Db(\mo A_{3}^{\otimes 5})\cong
	   \FS(F)$, the heart $\mo A_{3}^{\otimes 5}$, and central charges
	   $Z_{x}$ of quintic periods near the Gepner point $x=\infty$, we
	   have stability conditions on the heart such that each stable
	   object is isomorphic to a Lagrangian vanishing cycle and an
	   equivariant coherent sheaf of the Beilinson basis with a shift.
	  \end{theorem}

	  As its explicit forms recalled in the proofs of Theorems
	  \ref{thm:sheaf} and \ref{thm:lag}, we have period vectors
	  $\Pi_{B}(x)$ and $\Pi^{\infty}_{B}(x)$ consisting of quintic
	  periods near the large complex structure limit and the Gepner
	  point and we have the connection matrix $N$ for the analytic
	  continuation such that $\Pi_{B}(x)=N\Pi^{\infty}_{B}(x)$.

	  On stability conditions, we have dilation and rotation given
	  by multiplications on central charges \cite{Oka06a}.  For
	  example, gauge freedom \cite{KleThe} gives multiplications on
	  central charges.  This gives wall-crossings of second kind for
	  free.

	  On numerical predictions of the mirror symmetry, stability
	  conditions of Theorem \ref{thm:sheaf} are important. However,
	  stability conditions in Theorem \ref{thm:lag} give
	  wall-crossings as follows.

   	  \begin{corollary}\label{cor:analytic}
	   Stability conditions in Theorem \ref{thm:lag} deform into
	   ones in Theorem \ref{thm:sheaf} with wall-crossings of second
	   kind, which are different from ones of dilation and rotation.
	  \end{corollary}

	  Though we are taking components of period vectors simply as
	  some complex functions as per the last claim in Section
	  \ref{sec:claims}, but under the mirror symmetry, quotients
	  $\frac{\Pi_{B}(x)[i]}{\Pi_{B}(x)[1]}$ of the components are
	  coordinates of complexified K\"{a}hler classes of the
	  quintic. So, the following corollary is exactly as expected by
	  the mirror symmetry and the theory of Bridgeland stability
	  conditions.

	  \begin{corollary}\label{cor:analytic_deform}
	   For stability conditions in Theorem \ref{thm:sheaf},
	   quotients $\frac{\Pi_{B}(x)[i]}{\Pi_{B}(x)[1]}$ give
	   wall-crossings of second kind, which are different from ones
	   of dilation and rotation.
	  \end{corollary}

	   In Corollary \ref{cor:analytic_deform}, we would like to
	   clarify how quotients of quintic periods give non-trivial
	   properties of stability conditions, but Corollary
	   \ref{cor:analytic_deform} can be a part of Corollary
	   \ref{cor:analytic}, if we include local deformation.
	   
	   Let us discuss the mirror map $t(x)$\cite{CdGP}:
	    \begin{equation}\label{eq:mirror_map}
	     t(x)=\frac{\Pi_{B}(x)[2]}{\Pi_{B}(x)[1]}.
	    \end{equation}

	    We define stability conditions, which are, by Lemma
	    \ref{lem:sheaf_modified}, asymptotically the same as ones in
	    Theorem \ref{thm:sheaf}.  We also call these stability
	    conditions as stability conditions near the large complex
	    structure limit. We have the following corollary.

    	    \begin{corollary}\label{cor:mirror_map}
	     For stability conditions near the large complex structure
	     limit in Lemma \ref{lem:sheaf_modified}, by dilation and
	     rotation with the quintic period $w(x,0)$, the sum of
	     distinct central charges of stable objects is the mirror map.
	    \end{corollary}

	   \subsection{An application in the theory 
    of modular forms and periods}\label{sec:quasi_intro}

	   For a stability condition of $\Db_{G}(\Coh X)$, let us define
	   an object of $\Db(\Coh X)$ to be stable if its equivariant
	   object of $\Db_{G}(\Coh X)$ is stable. Then, for stability
	   conditions in Theorems \ref{thm:sheaf} and \ref{thm:lag}, we
	   obtain a stable spherical object \cite{SeiTho} of a
	   3-Calabi-Yau category $\Db(\Coh X)$ and motivic
	   Donaldson-Thomas invariants of the quantum dilogarithm
	   \cite{FadKas,KonSoi08,Kel,Qiu}.

	   Though generally expected, attaching a quasimodular or modular
	   form to quintic periods has its intrinsic difficulty, since
	   monodromy actions on quintic periods are not compatible with
	   ones of the modular group \cite{Zag12}. However, for motivic
	   Donaldson-Thomas invariants, we have the motivic variable $q$.

	   For $q=e^{2\pi i \tau}$, let $G_{2}(\tau)$ be the second
	   Eisenstein series defined as follows:
	   \begin{equation*}
	    G_{2}(\tau)=(2\pi
	     i)^{2}\left(-\frac{1}{24}+\sum_{m,r>0} m q^{mr}\right),
	   \end{equation*}
	   which is a quasimodular form \cite{KanZag} and a mock modular
	   form \cite{Zag06}.  As an analog of \cite[Theorem 1.5]{MelOka}
	   for Calabi-Yau surfaces such as K3 surfaces and the cotangent
	   bundle of $\bP^1$, we have the following.

	   \begin{theorem}\label{thm:quasimodular}
	    For each stability condition of central charges of quintic
	    periods in Theorems \ref{thm:sheaf} and \ref{thm:lag} and
	    Lemma \ref{lem:sheaf_modified}, motivic Donaldson-Thomas
	    invariants of each stable object of $\Db(\Coh X)$ give the
	    quasimodular form $G_{2}(\frac{\tau}{2})-G_{2}(\tau)$.
	   \end{theorem}

	  \subsection{Relations to existing works}

	  With complexified K\"ahler classes under the correspondences
	  in the last claim in Section \ref{sec:claims}, Douglas
	  \cite[Section 4]{Dou02} argued that approximations of central
	  charges of periods near the large complex structure limit are
	  given by Mukai vectors (cf. \cite[Appendix E]{HonOku}), by
	  which Bridgeland found the profound application \cite{Bri08}
	  for K3 surfaces. Along this line but without $\sqrt{\Todd X}$
	  as explained in \cite[Section 1.4]{BMT}, constructing
	  Bridgeland stability conditions for the quintic has been a
	  significant conjecture \cite{BMT,BBMT}.

	  For non-projective cases, central charges of the
	  Gel'fand-Kapranov-Zelevinski system (GKZ for short) system of
	  the $A_{1}$ singularity \cite{Hos04} and Bridgeland stability
	  conditions for the cotangent bundle of $\bP^1$
	  \cite{Tho,Bri05,Oka06b} have already been discussed in
	  \cite{Tho}.  We have its HMS \cite{IUU}. For the cotangent
	  bundle and a closely related case, in Section \ref{sec:a1}, we
	  confirm analogs of our statements for the quintic.  For the
	  local $\bP^2$ \cite{BayMac,CCG}, though its HMS is still a
	  conjecture, similar statements are expected to hold by
	  Lagrangians.

	 For the quintic, there are other HMS with Novikov rings
	 \cite{NohUed,She}.  We have deformations of stability
	 conditions whose parameter is of the Novikov rings and is of
	 solutions of the Picard-Fuchs equation.

	 Though we mainly focus on the quintic for its significance in
	 the mirror symmetry, we expect that similar statements hold for
	 other Fermat Calabi-Yau varieties by HMS \cite{Oka09,FutUed} and
	 suitable central charges.

	\section{HMS}\label{sec:HMS}
	Let $A_{n}$ be the type-$A$ Dynkin quiver with $n+1$ vertices and
	one-way arrows. Let us recall the following HMS
	\cite{Oka09,FutUed}:
	 \begin{equation}\label{eq:hms}
	  \Db_{G}(\Coh X)\cong \Db(\mo A_{3}^{\otimes 5})\cong \FS(F).
	 \end{equation}

	 In this article, vertices of the quiver $A_{3}^{\otimes 5}$ are
	 indexed by tuples $s=\{s_{1},s_{2},s_{3},s_{4},s_{5}\}$ for
	 $0\leq s_{i}\leq 3$ such that the source and sink vertices are
	 indexed by $\{0,0,0,0,0\}$ and $\{3,3,3,3,3\}$. We have
	 commuting relations on arrows of the quiver $A_{3}^{\otimes 5}$
	 such that the composition of arrows $s\to s'\to s''$ for
	 $s_{i}'=s_{i}+1$ and $s_{j}''=s_{j}'+1$ and $1\leq i,j\leq 5$ is
	 equal to that of $s\to s'''\to s''$ for $s_{j}'''=s_{j}+1$ and
	 $s_{i}''=s_{i}'''+1$.  For vertices $s$ of $\sum s_{i}\leq 2$,
	 we have the following partial figure of the quiver
	 $A_{3}^{\otimes 5}$.
	 \begin{figure}[H]
	  \begin{center}
	   \def\svgwidth{160pt}
	   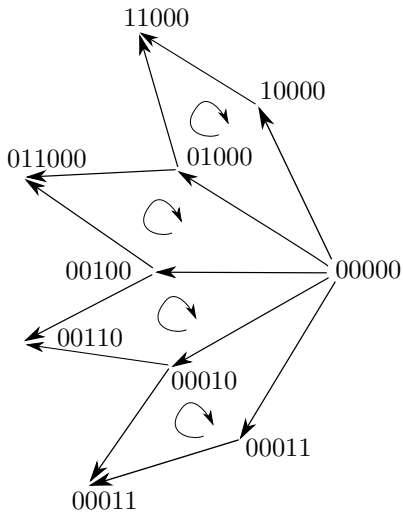
	   \caption{The partial figure of the quiver $A_{3}^{\otimes 5}$}
	  \end{center}
	 \end{figure}
	 For each vertex $s$, the simple representation of $\mo
	 A_{3}^{\otimes 5}$ with the one-dimensional complex vector
	 space at the vertex is also denoted by $s$.

       	 The reader can consult \cite{Aur} for an introduction to
       	 Fukaya-type categories.  We have a morsification of
       	 $x_{i}^{5}:\bC\to \bC$ by A'Campo as discussed in \cite{Sei01}
       	 and its products \cite{AKO}.  For such a morsification of $F$,
       	 which we keep taking in the following, simple representations
       	 of $\mo A_{3}^{\otimes 5}$ and their Ext-algebra correspond to
       	 Lagrangian vanishing cycles in the regular zero locus of the
       	 morsification and their Lagrangian Floer theory in the
       	 formulation of Fukaya-Seidel categories. This is not very
       	 difficult to see, since the Ext algebra is formal as a Koszul
       	 algebra \cite{ConGoe}.  The Ext algebra of simples
       	 representations of the quiver $A_{3}$ is Koszul and tensor
       	 products of Koszul algebras are Koszul \cite{Zac}.

       	 Let us recall that for $\Db(\Coh X)$, we have the
       	 autoequivalence $\tau$ of the monodromy around the Gepner
       	 point; namely, for the spherical twist $T_{\cO_{X}}$ of
       	 $\cO_{X}$ \cite{SeiTho} and each object $E\in \Db(\Coh X)$, we
       	 have
       	 \begin{equation}\label{eq:Gepner}
	  \tau(E)\cong \cO(1)\otimes T_{\cO_{X}}(E)
	 \end{equation}
       	 and $\tau^{5}\cong [2]$.  

       	 For the $n$-th exterior product of the cotangent bundle of
       	 $\bP^{4}$ restricted to the quintic, denoted by $\Omega^{n}$, we
       	 have that $\cO_{X}[3], \tau(\cO_{X}[3])\cong \cO_{X}(1)[1],
       	 \tau^2(\cO_{X}[3])\cong \Omega(2)[2], \tau^3(\cO_{X}[3])\cong
       	 \Omega^2(3)[3], \tau^4(\cO_{X}[3])\cong \Omega^3(4)[4]$
       	 \cite{Asp}.  Objects $\tau^{i}(\cO_{X})[-i]$ are the Beilinson
       	 basis of $\Db(\Coh X)$.  With $\Ext^1$ arrows and the dashed arrow
       	 indicating $[-2]$, we have the following quintic quiver
       	 \cite{DGJT}:
       	 \begin{figure}[H]
	  \begin{center}
	   \def\svgwidth{250pt}
	   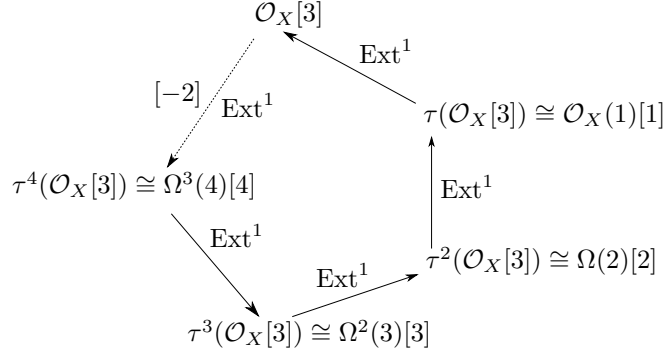
	   \caption{Quintic quiver}
	  \end{center}
	 \end{figure}

       By the Koszul duality, the Ext-algebra of simple representations
       of $\mo A_{3}^{\otimes 5}$ is anti-commutative.  In terms of the
       tensor product of graded matrix factorizations of
       $x_{i}^{5}:\bC\to \bC$ for $1\leq i \leq 5$ \cite{Orl09,KST},
       taking advantages of computations in \cite{ADD}, we see that we
       have objects corresponding to simple representations of $\mo
       A_{3}^{\otimes 5}$. Then, by forgetting the multi-grading ($G$
       equivariance) of the tensor product, simple representations $s$
       give objects $\tau^{-\sum s_{i}}(\cO_{X})$ in the category of
       graded matrix factorizations of $F$ \cite{Orl09}.  Let us mention
       that $\tau$ coincides with the grade shift of the category of
       graded matrix factorizations of $F$.

     \section{Stability conditions}\label{sec:stab}

     Let $\bfH$ be the union of the upper-half plane and the negative
     real line not including the zero.  Let us recall that if the central
     charge of each non-zero object of the category of representations of
     a quiver is in $\bfH$, we always have a Bridgeland stability
     condition on the category.

     We formulate a notion of stability conditions for cases when central
     charges of non-zero objects of an abelian category are not
     necessarily in $\bfH$. For us, this reflects the fact that for
     non-zero objects of the heart in Theorems \ref{thm:sheaf} and
     \ref{thm:lag}, we have quintic periods which are asymptotically some
     several powers of $\frac{\log(x)}{2\pi i}$; this is not the case of
     Remark \ref{rmk:A1}.

     With our formulation of stability conditions, when central charges
     of non-zero objects are not contained in $\bfH$, in general, it is
     highly non-trivial that whether we have a stability condition even
     on the category of representations of a quiver.

     From \cite{Bri07}, let us recall that to give a Bridgeland stability
     condition on a triangulated category for a given family of
     semistable objects is equivalent to give a heart of the triangulated
     category and a central charge on the heart with the
     Harder-Narasimhan property recalled in the following.

       \begin{definition} \label{def:stab}
	For a triangulated category $\cT$ and the Grothendieck group
	$K(\cT)$, a stability condition consists of a heart $\cA$ of the
	triangulated category $\cT$ and a central charge $Z\in
	\Hom_{\bZ}(K(\cT), \bC)$ with the following:
	\begin{enumerate}
	 \item We have {\it semistable objects} $Q\in \cA$ such that 
	       $Z(Q)=m(Q)\exp(\phi_{Q}i)$ for some {\it masses} $m(Q)>0$ and 
	       {\it phases} $\phi_{Q}\in \bR$.
	 \item If $\Ext^{1}(Q,Q')\not\cong 0$ for semistable objects $Q,Q'\in
	       \cA$, then $|\phi_{Q}-\phi_{Q'}|< \pi$.
	 \item If $\phi_{Q'}>\phi_{Q}$, then $\Hom(Q',Q)\cong 0$.
	 \item For each nonzero object $E\in \cA$, we have semistable
	       objects $Q_{i}\in \cA$ such that
	       $\phi_{Q_{i+1}}>\phi_{Q_{i}}$ with the following filtration
	       by short exact sequences:
	       \[
	       \xymatrix@C=.5em{
	       0_{\
	       }\ar@{=}[r]&E_n\ar@{->}[rr]&&E_{n-1}\ar@{->}[rr]\ar[dl]&&E_{n-2}\ar@{->}[rr]\ar[dl]
	       &&\ldots\ar@{->}[rr]&&E_{1}\ar[dl]\ar@{->}[rr]&&E_{0}\ar[dl]\ar@{=}[r]&E_{\ }\\
	       && Q_{n-1}&&Q_{n-2}&&&&Q_{1}&&Q_0 
	       }
	       \]
	\end{enumerate}
	For each phase, semistable objects which can not be obtained by
	non-trivial extensions of semistable objects of the phase are
	called {\it stable}. The above filtration is called {\it the
	Harder-Narasimhan filtration} of the object $E$ and having such
	filtrations for non-zero objects of the heart, originally for
	central charges of semistable objects in $\bfH$, is called the
	Harder-Narasimhan property of the central charge on the heart.
       \end{definition}

	By the third and the fourth conditions in Definition
	\ref{def:stab}, we have the uniqueness of Harder-Narasimhan
	filtrations of each non-zero object up to isomorphisms \cite{GKR}.

	\begin{remark}\label{rmk:stab_def}
	 For a stability condition on a heart $\cA$, we have tiltings of
	 the heart $\cA$ as in {\it loc cite}. For each $\phi\in \bR$, we
	 have the heart $\cA^{\sigma}_{\phi}$ as the extension-closed
	 full subcategory of $\cT$ consisting of semistable objects
	 $Q''\in \cA$ such that $\phi_{Q''}>\phi$ and objects $Q''[1]$ of
	 semistable objects $Q''\in \cA$ such that $\phi_{Q''}\leq \phi$.
	\end{remark}

	 For Bridgeland stability conditions, the second condition in
	 Definition \ref{def:stab} holds automatically for a heart.  It
	 is a simple condition to restrict orders of possible phases of
	 semistable objects for a given central charge. Even when do not
	 have Bridgeland stability conditions, we still have tiltings of
	 certain stability conditions such as of Theorems \ref{thm:sheaf}
	 and \ref{thm:lag}.

	 \begin{proposition}\label{prop:exceptional}
	  Let us assume that for a stability condition $\sigma$ of a
	  heart $\cA$, a central charge $Z$, phases $\phi^{\sigma}_{Q}$,
	  and semistable objects $Q,Q'\in \cA$, we have
	  $\Ext^{2}(Q,Q')\cong 0$ for
	  $\phi^{\sigma}_{Q}>\phi^{\sigma}_{Q'}$ and $\Hom(Q,Q')\cong 0$
	  for $\phi^{\sigma}_{Q}<\phi^{\sigma}_{Q'}$. In addition, let us
	  assume that phases are bounded.  Then, we have a stability
	  condition $\sigma'$ on the titled heart $\cA^{\sigma}_{\phi}$
	  (in the notation of Remark \ref{rmk:stab_def}) for each $\phi\in
	  \bR$ with the central charge $Z$ and semistable objects $Q\in
	  \cA^{\sigma}_{\phi}$ for semistable objects $Q\in \cA$ such that
	  $\phi^{\sigma}_{Q}> \phi$ and semistable objects $Q[1]\in
	  \cA^{\sigma}_{\phi}$ for semistable objects $Q\in \cA$ such that
	  $\phi^{\sigma}_{Q}\leq \phi$.
	 \end{proposition}
	 \begin{proof}
	  Notice that for semistable objects $Q,Q'\in \cA$ such that
	  $\phi^{\sigma}_{Q}>\phi^{\sigma}_{Q'}$, we have $
	  \Ext^{1}(Q,Q'[1])\cong \Ext^{2}(Q,Q')\cong 0$ and
	  $\Ext^{1}(Q'[1],Q)\cong 0$ by the first assumption.  We take
	  phases $\phi^{\sigma'}_{Q[1]}$ of $Q[1]\in \cA^{\sigma}_{\phi}$
	  for semistable objects $Q\in \cA$ so that
	  $\phi^{\sigma'}_{Q[1]}=\phi^{\sigma}_{Q}+(2n+1)\pi$ for some
	  $n\geq 0$ and the range of phases $\phi^{\sigma'}_{Q[1]}$ do
	  not overlap with that of phases $\phi^{\sigma'}_{Q}$ of other
	  semistable objects of $Q\in \cA^{\sigma}_{\phi}$.
	 \end{proof}

	 We assume $K(\cT)$ is of a finite rank; if not, we replace
	 $K(\cT)$ with its some quotient of a finite rank such as
	 numerical Grothendieck groups \cite{Bri07}.  We consider
	 deformation of central charges in the space $\Hom_{\bZ}(K(\cT),
	 \bC)$, which has the standard topology by the finiteness of
	 $K(\cT)$, and deformation of stability conditions.

	 For a stability condition on a heart $\cA$ and an interval
	 $(a,b)$, let $\cP(a,b)$ be the extension-closed full subcategory
	 of the heart $\cA$ consisting of semistable objects $E$ such
	 that $\phi_{E}\in (a,b)$.

	 \begin{remark}\label{rmk:spaces}
	  With a care on the second condition in Definition
	  \ref{def:stab}, deformation theory of \cite{Bri07} on
	  Bridgeland stability conditions on a heart naturally extends to
	  our cases.  We assume the local-finiteness in {\it loc cite}.
	  In addition, we assume that for each phase $\phi_{Q}$ of a
	  semistable object $Q$ of a heart $\cA$, we have
	  $\epsilon_{\phi_{Q}}> 0$ with the following condition: for
	  phases $\phi_{Q},\phi_{Q'}$ of semistable objects $Q,Q'$ of the
	  heart $\cA$ and objects $E\in \cP(\phi_{Q}-\epsilon_{\phi_{Q}},
	  \phi_{Q}+\epsilon_{\phi_{Q}})$ and $E'\in
	  \cP(\phi_{Q'}-\epsilon_{\phi_{Q'}},
	  \phi_{Q'}+\epsilon_{\phi_{Q'}})$, if $\Ext^{1}(E,E') \not\cong
	  0$ then
	  $|\phi_{Q}-\phi_{Q'}|+\epsilon_{\phi_{Q}}+\epsilon_{\phi_{Q'}}
	  <\pi$.  Both assumptions easily hold for stability conditions
	  in Theorems \ref{thm:sheaf} and \ref{thm:lag}.  For a small
	  open interval contained in
	  $(\phi-\epsilon_{\phi},\phi+\epsilon_{\phi})$ of some phase
	  $\phi$ of a semistable object, we consider deformation of
	  central charges such that central charges of semistable objects
	  whose phases are in the interval stay in the interval.
	  Moreover, we can take small non-overlapping open intervals such
	  that each of them contained in
	  $(\phi-\epsilon_{\phi},\phi+\epsilon_{\phi})$ of some phase
	  $\phi$ of a semistable object and consider simultaneous
	  deformations of central charges.  On the extension-closed full
	  subcategory consisting of semistable objects of phases in the
	  interval, we apply the deformation theory of {\it loc cite}.
	 \end{remark}

	 In Remark \ref{rmk:spaces}, we put deformation of stability
	 conditions on a heart. However, even when we do not have
	 Bridgeland stability conditions, under the assumptions of
	 Proposition \ref{prop:exceptional}, we can tilt the heart with
	 stability conditions and work on deformations of stability
	 conditions on the tilted heart.  For stability conditions in
	 Theorems \ref{thm:sheaf} and \ref{thm:lag}, we put discussions
	 on their specific deformations in Remarks \ref{rmk:deform} and
	 \ref{rmk:Teichmuller}.

	 For our later reference, let us put the following lemma on
	 wall-crossings of second kind for certain deformation of
	 stability conditions, which involve no change of orders of
	 semistable objects but with deformation of central charges.

	  \begin{lemma}\label{lem:deform}
	   For a heart $\cA$, let us assume that we have stability
	   conditions $\sigma,\sigma'$ with central charges $Z$ and $Z'$
	   and phase $\phi^{\sigma},\phi^{\sigma'}$.  In addition, let us
	   assume that for $\sigma,\sigma'$, we have the same collection
	   of semistable objects such that for semistable objects
	   $Q_{1},Q_{2}$, we have
	   $\phi^{\sigma'}_{Q_{1}}<\phi^{\sigma'}_{Q_{2}}$ if and only if
	   $\phi^{\sigma}_{Q_{1}}<\phi^{\sigma}_{Q_{1}}$.  Then, for some
	   $\phi$ and a semistable object $Q$ such that
	   $\phi^{\sigma}_{Q}<\phi<\phi^{\sigma'}_{Q}$, we have a
	   non-trivial wall-crossing of second kind between $\sigma$ and
	   $\sigma'$.
	  \end{lemma}
	  \begin{proof}
	   In the notation of Remark \ref{rmk:stab_def}, hearts
	   $\cA^{\sigma}_{\phi}$ and $\cA^{\sigma'}_{\phi}$ differ.
	  \end{proof}

	  \section{Statements and proofs}\label{sec:thms}
	  For each simple representation $s$ of the heart $\mo
	  A_{3}^{\otimes 5}$ and $[1:x]\in \bP^1$, by the definition of
	  Equation \ref{central} and HMS of Equation \ref{eq:hms}, we
	  have the following:
	  \begin{equation*}
	   Z_{x}(s)=Z_{x}(\tau^{-\sum s_{i}}(\cO_{X})).
	  \end{equation*}

	  Let us recall that near the large complex structure limit and
	  the Gepner point, we have quintic periods such that they have
	  finite radii of convergence and make period vectors
	  $\Pi_{B}(x)$ and $\Pi_{B}^{\infty}(x)$.

	  Let us prove Theorem \ref{thm:sheaf}.

	  \begin{proof}
	   Let us recall quintic periods in the form of \cite[Example
	   1]{Hos00} as follows:
	\begin{align*}
	 w^{(0)}(x)&=w(x,0),\\
	 w^{(1)}(x)&=\frac{1}{2\pi i}\frac{\partial}{\partial p}w(x,p)\mid_{p=0},\\
	 w^{(2)}(x)&=\frac{1}{2! (2 \pi i)^2} \ 5 \ \frac{\partial^2}{\partial
	  p^2}w(x,p)\mid_{p=0} +\frac{11}{2}\frac{1}{2\pi
	  i}\frac{\partial}{\partial p}w(x,p)\mid_{p=0},\\
	 w^{(3)}(x)&=-\frac{1}{3! (2\pi i)^3} \ 5 \
	  \frac{\partial^3}{\partial     p^3}w(x,p)\mid_{p=0}
	 -\frac{1}{2 \pi i} \ \frac{50}{12} \ \frac{\partial}{\partial
	 p}w(x,p)\mid_{p=0}.
	\end{align*}
	Then we have the period vector around the large complex structure
	limit $\Pi_{B}(x)={}^{t}(w^{(0)}(x), w^{(1)}(x), w^{(2)}(x),
	w^{(3)}(x))$.  For the monodromy around the Gepner point, on
	$\Pi_{B}(x)$, in the form of \cite[Example 4]{Hos00}, we have 
	the following monodromy matrix:
	\[ M_{\infty}= \left( \begin{array}{cccc}
			1& 0&0 &-1 \\
		       -1& 1&0 & 1 \\
		       -3&-5&1 & 3 \\ 5&-8&1 &-4 \end{array} \right).\]

	   Notice that $M_{\infty}$ corresponds to $\tau^{-1}$, since for
	   monodromy matrices $M_{0}$ and $M_{1}$ corresponding to
	   $T_{\cO_{X}}$ and $\otimes \cO_{1}$ in {\it loc cite}, we have
	   that $M_{1}M_{0}=M_{\infty}^{-1}$ corresponds to $\tau$
	   (cf. \cite[Section 4.2]{Hor}).

	We have $\Pi_{B}^{\infty}(x)[4]=Z_{x}(\cO_{X})$.

	For $0<x<<1$, we have
	\begin{equation*} \label{eq:mirror}
	 \Pi_{B}(x)=
	 \left(\begin{array}{c}
	  1+O(x)\\
		\frac{\log(x)}{2\pi i}+O(x)\\
		\frac{5}{2}\left(\frac{\log(x)}{2\pi
			    i}\right)^2+O(\log(x))\\
		-\frac{5}{6} \left(\frac{\log(x)}{2\pi i}\right)^3+O(\log(x))
	 \end{array}\right)
	 \footnote{Numbers $\frac{5}{2}$ and $-\frac{5}{6}$ are in Equation
	 5.5 in \cite{CdGP} and Equation 5.4 in \cite{KleThe} up to a 
	 constant. }.
       \end{equation*} 
	For $\Pi_{B}(x)':={}^{t}(0,0,m_{1},m_{0})= {}^{t}(0,0,
	-\frac{5}{8\pi^2}\log(x)^2, -\frac{5i}{48\pi^3} \log(x)^3)$, we
	have the following:
       \begin{align*}
	\Pi_{B}^{\infty}(x)'[4]&= m_{0},\\
	M_{\infty}\Pi_{B}^{\infty}(x)'[4]&= m_{1}-4 m_{0},\\
	M_{\infty}^2\Pi_{B}^{\infty}(x)'[4]&= -3 m_{1}+ 6 m_{0},\\
	M_{\infty}^3\Pi_{B}^{\infty}(x)'[4]&=3 m_{1}- 4 m_{0},\\
	M_{\infty}^4\Pi_{B}^{\infty}(x)'[4]&= - m_{1}+ m_{0}.
	\end{align*}

	   So comparing ratios of coefficients, for $0<x<<1$, we can take
	   the difference
	$a_{i}=\Arg(M_{\infty}^{i+1}\Pi_{B}(x)[4])-\Arg(M_{\infty}^{i}\Pi_{B}(x)[4])$
	   to be
       \begin{align*}
	\pi>a_{0}>a_{1}>a_{2}>a_{3}>0
	\end{align*}
	with each $a_{i}$ being close to  $\pi$.

	   Let us take simple representations $s$ of the heart $\mo
	   A_{3}^{\otimes 5}$ as stable objects and let angles of $Z_{x}(s)$
	   be their phases in the increasing way approximately by $\pi$ as
	   $\sum a_{i}$ increases.

	   It is clear from the standard theory of quiver
	   representations that these stable objects with their
	   self-direct sums taken as semistable objects give stability
	   conditions. It is clear that the second condition in
	   Definition \ref{def:stab} holds, since
	   $\Ext^{1}(s,s')\not\cong 0$ only when we have an arrow $s\to
	   s'$.
	  \end{proof}

	  Even if we weaken the second condition in Definition
	  \ref{def:stab} to be $\phi_{Q}-\phi_{Q'}<\pi$ for semistable
	  objects $Q,Q'\in \cA$ with $\Ext^{1}(Q,Q')\not\cong 0$ (and
	  weaken the second assumption in Remark \ref{rmk:spaces} with
	  the inequality
	  $\phi_{Q}-\phi_{Q'}+\epsilon_{\phi_{Q}}+\epsilon_{\phi_{Q'}}<\pi$),
	  a stability condition at the large complex structure limit
	  would not exist, unless we take dilation and rotation.  For
	  small $x$, $F_{i}=|M_{\infty}^{i}\Pi(x)[4]|$ go to infinitely
	  large at similar rates.  Let us mention that for central
	  charges $\Pi(x)[1]$, which are of a skyscraper sheaf of the
	  quintic, $\frac{\left|\Pi(x)[1]\right|}{F_{i}}$ goes to zero
	  as $x\to 0$.

	  Let us prove Theorem \ref{thm:lag}.

       \begin{proof}
	For $k=1,\cdots, 5$ and $j=0,1,\cdots, 4$, let us recall quintic
	periods in the form of \cite[Example 4]{Hos00} as follows:
	\begin{align*}
	 \tilde{\omega_{k}}(x)&=-\frac{1}{5}\frac{1}{(2\pi
	 i)^4}\sum_{N=0}^{\infty}
	\frac{\Gamma(N+\frac{k}{5})^{5}}{\Gamma(5N+k)} x^{-N-\frac{k}{5}}, \\
	 \omega_{j}^{\infty}(x)
	 &=\sum_{k=1}^{5}(1-\xi^k)^4\xi^{kj}\tilde{\omega_{k}}(x).
	\end{align*}
	Then $\Pi_{B}^{\infty}(x)={}^{t}(\omega_{0}^{\infty}(x),
	\omega_{1}^{\infty}(x), \omega_{2}^{\infty}(x),
	\omega_{4}^{\infty}(x))$ is a period vector around the Gepner point.

	In {\it loc cite}, we have the following connection matrix
	\[ N= \left( \begin{array}{cccc}
	       1&0&0&0 \\
		      -\frac{2}{5}&\frac{2}{5}&\frac{1}{5}&-\frac{1}{5} \\
		      -\frac{21}{5}&\frac{1}{5}&\frac{3}{5}&-\frac{8}{5} \\
		     1&-1 &0&0 \end{array} \right),\] such that
	$N\Pi_{B}^{\infty}(x)[4]=Z_{x}(\cO_{X})$ and the monodromy matrix of the
	period vector around the Gepner point as follows:
	\[ M= \left( \begin{array}{cccc}
	       0&1&0&0 \\
	       0&0&1&0 \\
	      -1&-1&-1&-1 \\
	       1&0 &0&0  \end{array} \right).\] 

	By changing the variable $x\mapsto x^{-5}$,
	$\omega_{j}^{\infty}(x)=\omega_{0}^{\infty}(\xi^j x)$.  Up to a
	common constant factor on each $\omega_{j}^{\infty}(x)$,
	first-order approximations of $\omega_{j}^{\infty}(x)$ for
	$0<x<<1$ are $\xi^{j} x$. For $\Pi_{B}^{\infty}(x)':={}^{t}(x,\xi
	x, \xi^2 x,\xi^4 x)$, we have
	\begin{align*}
	 NM^{i}\Pi_{B}^{\infty}(x)'[4]&= \xi^{i}(1-\xi) x.
	\end{align*}

	So for $0<x<<1$, we can take 
	\begin{align*}
	 \Arg(N\Pi_{B}^{\infty}(x)[4])&<
	 \Arg(NM\Pi_{B}^{\infty}(x)[4])<\\
	 \Arg(NM^2\Pi_{B}^{\infty}(x)[4])
	 &<\Arg(NM^3\Pi_{B}^{\infty}(x)[4])<\\
	 \Arg(NM^4\Pi_{B}^{\infty}(x)[4])
	 &< \Arg(N\Pi_{B}^{\infty}(x)[4])+2\pi
	\end{align*}
	with each difference close to $\frac{2\pi}{5}$.

	Let our stable objects be simple representations $s$ of the heart
	$\mo A_{3}^{\otimes 5}$ and their phases be angles of $Z_{x}(s)$
	in the increasing way approximately by $\frac{2\pi}{5}$ as $\sum
	a_{i}$ increases.
       \end{proof}

       A stability condition at the Gepner point would not exist, unless
       we take dilation and rotation.  Masses of central charges of
       simple representations $s$ go to zero as $x\to \infty$. The notion
       of stability conditions whose central charges have the pentagon
       symmetry has been discussed in \cite{Oka09,Tod}.

       For simple representations $s^{i}$ such that $\sum_{j}
       s^{i}_{j}=i$, we have the following figures of central charges
       near the large complex structure limit and the Gepner point by
       direct computations of quintic periods for $x=10^{10}$ and
       $x= 10^{-10}$.
       \begin{figure}[H]
	\minipage{0.48\textwidth}
	\def\svgwidth{160pt}
	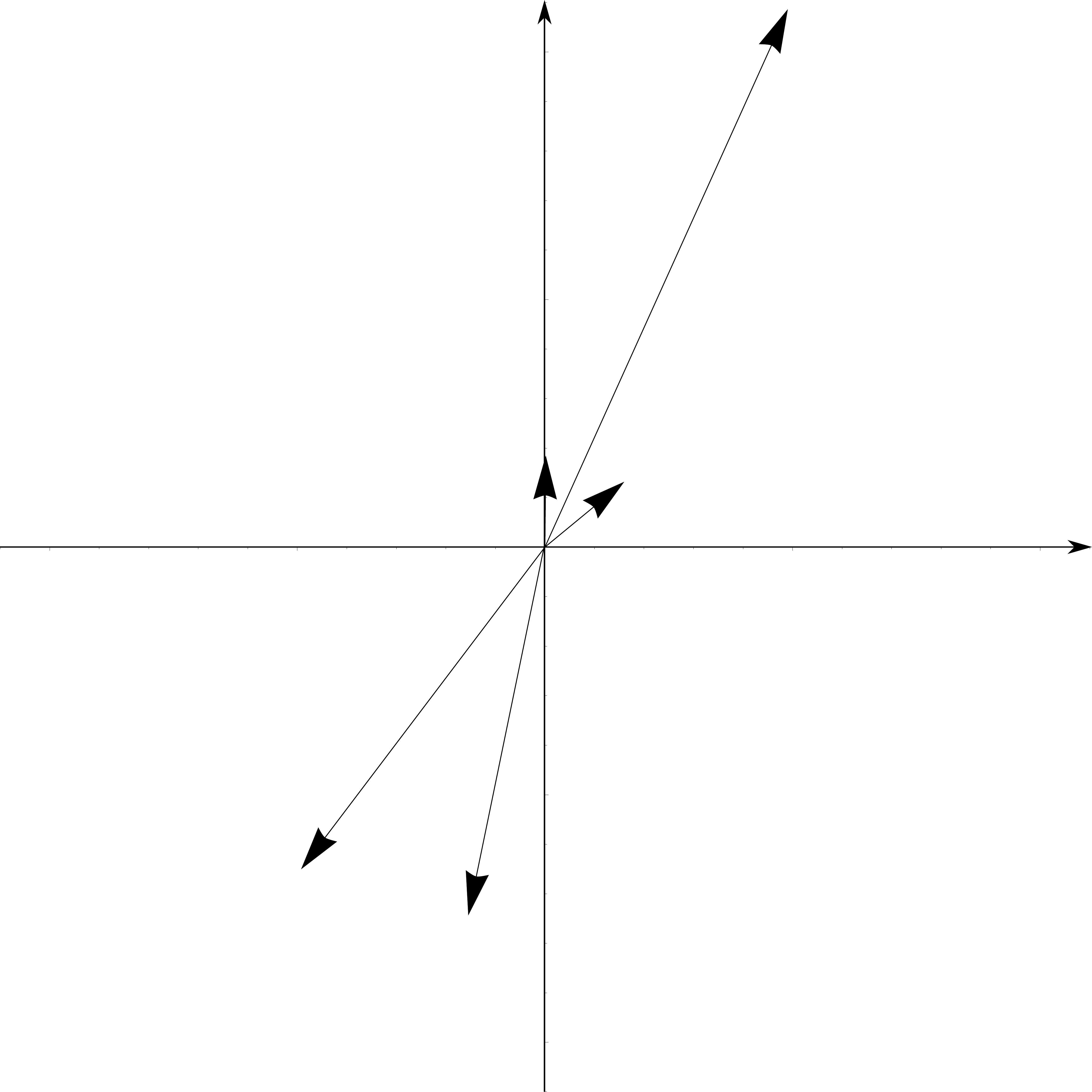
	\caption{Central charges near the large complex structure limit}\hfill
	\endminipage
	\minipage{0.48\textwidth}
	\def\svgwidth{150pt}
	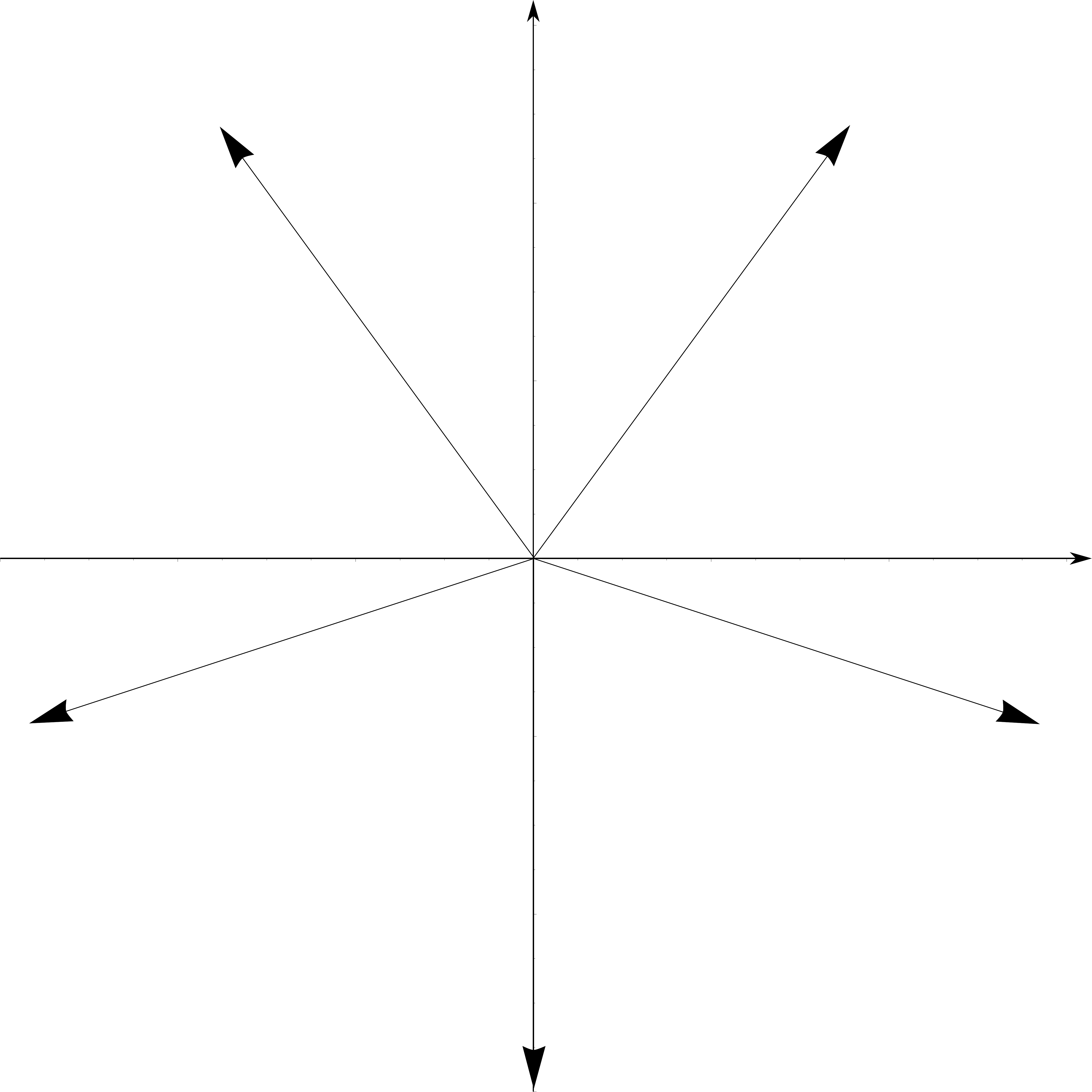
	\caption{Central charges near the Gepner point}
	\endminipage
       \end{figure}
       Let us prove Corollary \ref{cor:analytic}.

       \begin{proof} 
	From phases of central charges of stable objects in the proofs of
	Theorems \ref{thm:sheaf} and \ref{thm:lag}, we can deform central
	charges so that we have desired phases and masses of stable
	objects. By Lemma \ref{lem:deform}, we have wall-crossings of
	second kind, which are different from ones of dilation and
	rotation.
       \end{proof}

	 Let us prove Corollary \ref{cor:analytic_deform}.
	 \begin{proof}
	  From the proof of Theorem \ref{thm:sheaf}, by Lemma
	  \ref{lem:deform}, deforming quintic periods near the large
	  complex structure limit by letting $x\mapsto 0$ for $x>0$ give
	  wall-crossings of second kind, which are different from ones of
	  dilation and rotation.
	 \end{proof}

	 As mentioned in the introduction, near the large complex
	 structure limit, let us define central charges $Z'_{x}$ as
	 follows:
	\begin{align*}
	 Z'_{x}(s)=M^{-\sum s_{i}}\Pi_{B}(x)[4]+
	 \frac{\Pi_{B}(x)[2]}{5}.
	\end{align*}
	We have the following lemma in order.
	\begin{lemma}\label{lem:sheaf_modified}
	 For $\Db_{G}(\Coh X)\cong \Db(\mo A_{3}^{\otimes 5})\cong
	 \FS(F)$, the heart $\mo A_{3}^{\otimes 5}$, and central charges
	 $Z'_{x}$ near the large complex structure limit $x=0$, we have
	 stability conditions on the heart such that each stable object
	 is isomorphic to a Lagrangian vanishing cycle and an equivariant
	 coherent sheaf of the Beilinson basis with a shift. These
	 stability conditions are asymptotically the same as ones in
	 Theorem \ref{thm:sheaf}.
	\end{lemma}
	\begin{proof}
	 From the proof of Theorem \ref{thm:sheaf}, it is clear that
	 adding $\frac{\Pi_{B}(x)[2]}{5}$ does not change asymptotics of
	 $\Pi_{B}(x)$ near the large complex structure limit.
	\end{proof}

	Let us prove Corollary \ref{cor:mirror_map}.
	\begin{proof}
	 Let $S(x)$ be the sum of distinct central charges of stable
	 objects.  Then $\sum_{i=0}^{4}M_{\infty}^{i}=0$.  By rotation
	 and dilation with the quintic period $w(x,0)=\Pi_{B}(x)[1]$, we
	 obtain the mirror map $t(x)=\frac{S(x)}{w(x,0)}$.
	\end{proof}

	\begin{remark}\label{rmk:numerical}
	 Before putting a remark on numerical predictions of
	 \cite{CdGP,BCOV,Giv,LLY,Zin} at the end of this paragraph, from
	 \cite[Appendix B]{Zin} let us recall the following.  In
	 \cite{CdGP}, for the variable $q=e^{2 \pi i t}$, Lambert
	 expansions of the normalized Yukawa coupling
	 $Y(q)=\frac{5}{(1-5^5 x)w(x,0)^2}\left( \frac{q}{x}
	 \frac{dx}{dq}\right)^{3}$ give the numerical prediction, which
	 is proved in \cite{Giv,LLY}. In \cite{BCOV}, by $\psi$ of
	 $x=(5\psi)^{-5}$, Lambert expansions of
	 $\partial_{t}F_{1}(\psi)$ of the generalized index
	 $F_{1}(\psi)=\ln\left(
	 \left(\frac{\psi}{w(5\psi^{-5},0)}\right)^{\frac{62}{3}}
	 \left(1-\psi^5\right)^{-\frac{1}{6}} \frac{d\psi}{dt}\right)$
	 give the numerical prediction, which is proved in \cite{Zin}.
	 Notice that, since we know asymptotics and linearity of quintic
	 periods, among quintic periods, $w(x,0)$ can be specified
	 geometrically by the space of stability condition of Lemma
	 \ref{lem:sheaf_modified}; in fact, for $S(x)$ in the proof of
	 Corollary \ref{cor:mirror_map}, $w(x,0)$ is the quintic period
	 such that the exponential function of the quotient of $2\pi i
	 S(x)$ by the quintic period is asymptotically the parameter $x$
	 of the space.  So, up to changes of variables of $x$, $t$, and
	 $w(x,0)$ by elementary functions, categorical and geometric
	 interpretations of factors of $Y(q)$ or $e^{F_{1}(\psi)}$ are
	 given by the space of stability conditions of Lemma
	 \ref{lem:sheaf_modified}, which are given by HMS.
	\end{remark}

	We have spaces of stability conditions in a general setting in
	Remark \ref{rmk:spaces}.  For stability conditions of
	$\Db_{G}(\Coh X)$, let us prove the following.

       \begin{corollary}\label{cor:PF}
	Near the Gepner point and the large complex structure limit, we
	have one-parameter subspaces of stability conditions of
	$\Db_{G}(\Coh X)$ such that any four of distinct central charges
	of stable objects give a basis of quintic periods.
       \end{corollary}
       \begin{proof}
	Period vectors $\Pi_{B}^{\infty}(x)$ and $\Pi_{B}(x)$ give bases
	of solutions of the Picard-Fuchs equation.  So the assertion
	follows from constructions of stability conditions in Theorems
	\ref{thm:sheaf} and \ref{thm:lag} and Lemma
	\ref{lem:sheaf_modified}.
       \end{proof}

       To be able to deform stability conditions in a certain way, in
       general we have to deal with highly non-trivial problems to
       carefully look into distributions of central charges of stable
       objects.

       For the quiver $A_{3}^{\otimes 5}$ and central charges of quintic
       periods, we have not seen {\it wall-crossings of first kind}
       \cite{KonSoi08}, which is a disappearance of a stable object by a
       deformation of central charges. Compared to this case, we have a
       wall-crossing of first kind as explained in Remark
       \ref{rmk:a_1-wall}.

       \begin{remark}\label{rmk:deform}
	Stability conditions in Theorems \ref{thm:sheaf} and
	\ref{thm:lag} and in Lemma \ref{lem:sheaf_modified} deform into
	Bridgeland stability conditions only with wall-crossings of
	second kind, by deforming central charges of stable objects into
	$\bfH$ without changing their orders.
       \end{remark}

       \begin{remark}\label{rmk:Teichmuller}
	Since we can deform a stability condition into Bridgeland
	stability conditions as in Remark \ref{rmk:deform}, by known
	results on Bridgeland stability conditions, we expect that we
	can deform stability conditions of Theorem \ref{thm:sheaf} and
	\ref{thm:lag} on the heart $\mo A_{3}^{\otimes 5}$ into
	stability conditions of central charges of quintic periods on a
	heart $\cA \cong \mo A_{3}^{\otimes 5}$ such that the heart
	$\cA$ consists of the object corresponding to a non-source
	vertex $s$ of the heart $\mo A_{3}^{\otimes 5}$ as the object
	corresponding to the source vertex of the heart $\cA$.  This
	would lead to further understanding of the {\it Teichm\"uller
	theory} discussed by Aspinwall-Douglas \cite{AspDou}, since by
	forgetting the equivariance, we can obtain the object
	corresponding to the vertex $s$ by applying $\tau^{-i}$ for some
	$i>0$ on the object corresponding to the source vertex of the
	heart $\mo A_{3}^{\otimes 5}$.  Theorems \ref{thm:sheaf} and
	\ref{thm:lag} and the {\it mirror conjecture} in \cite{Kon95}
	are local in nature.  Let us mention that by Proposition
	\ref{prop:exceptional}, we can tilt hearts with stability
	conditions without first passing to Bridgeland stability
	conditions.
       \end{remark}

   \section{A quasimodular form on quintic periods}\label{sec:broken_symmetry}  

   Let us prove Theorem \ref{thm:quasimodular}.  We recall that for a
   stability condition of $\Db_{G}(\Coh X)$, we have defined an object
   of $\Db(\Coh X)$ to be stable if its equivariant object in
   $\Db_{G}(\Coh X)$ is stable.

   Geometrically, in the following, we obtain a quasimodular form
   essentially by counting a single point, which is the moduli space of
   a spherical object, with multiplicities in a 3-Calabi-Yau category.

    \begin{proof}
     We obtain a stable spherical object in $\Db(\Coh X)$, since each
     object $\tau^{i}(\cO_{X})$ is spherical. For the $\Ext$ algebra of
     direct sums of the stable spherical object and the motivic
     parameter $q$, we have the quantum dilogarithm \cite{KonSoi08,Kel}:
     \begin{align*}
      \bfE(q^{\frac{1}{2}},z)=\sum_{m\geq
      0}\frac{(-q^{\frac{1}{2}})^{m^2}}{(q^m-1)\cdots(q^m-q^{m-1})}z^{m}
     \end{align*}
     (with the change of the variable $q^{\frac{1}{2}}$ in {\it loc cite}
     into $-q^{\frac{1}{2}}$).

     As in \cite{MelOka}, by taking the formal logarithm of the quantum
     dilogarithm with respect to the variable $z$, we have
     \begin{align*}
      \log_{z}
      \bfE(q^{\frac{1}{2}},z)=\sum_{m>0}\frac{1}{m}\frac{q^{\frac{m}{2}}}{1-q^{m}}z^{m}
      &=\sum_{m,r>0}\frac{1}{m}(q^{\frac{mr}{2}}-q^{mr})z^{m}.
     \end{align*}
     Let $q=\exp(2\pi i \tau)$ and $z=\exp(2\pi i \tau')$. By
     differentiating $\log_{z}\bfE(q^{\frac{1}{2}},z)$ twice with respect
     to the variable $\tau'$, we obtain
      \begin{align*}
       \left.\frac{\partial^{2}}{\partial \tau'^{2}}
       \log_{z} \bfE(q^{\frac{1}{2}},z)\right|_{z=1}
       =G_{2}\left(\frac{\tau}{2}\right)-G_{2}(\tau),
      \end{align*}
     which is a quasimodular form \cite{KanZag}.
    \end{proof}

     By the proof of Theorem \ref{thm:quasimodular}, the quasimodularity
     in Theorem \ref{thm:quasimodular} is a non-trivial property of the
     quantum dilogarithm.  As in \cite{KonSoi08,Kel,Qiu}, the well-known
     quantum pentagon identity of the quantum dilogarithm explains the
     wall-crossing of first kind of $\mo A_{1}$ in Section
     \ref{sec:a1_quiver} and so of full extension-closed subcategories
     of $\mo A_{3}^{\otimes 5}$ isomorphic to $\mo A_{1}$.

     Let us mention that we can recover the quantum dilogarithm from the
     generating function $\frac{\partial^{2}}{\partial
     \tau'^{2}}\log_{z}\bfE(q^{\frac{1}{2}},z)$ by taking integrations
     with respect to the variable $\tau'$ with appropriate boundary
     values and the formal exponential function with respect to the
     variable $z$.

     The variable $\tau$ is conjectured to be $\Omega$-background of
     field theories of $\bR^{4}$ (a toric parameter of Nekrasov's
     partition functions a.k.a. graviphoton background)
     \cite{DimGuk,Kan}.

     Notice that we do not have $2$ in front of $G_{2}(\tau)$ in Theorem
     \ref{thm:quasimodular} so that we have a modular form
     $G_{2}(\frac{\tau}{2})-2 G_{2}(\tau)$.  However, for a modular form
     on quintic periods, let $J(\tau)=G_{2}(\frac{\tau}{2})-G_{2}(\tau)$
     for the quasimodular form in Theorem \ref{thm:quasimodular}. Then,
     measuring the failure for $J(\tau)$ to be modular by taking
     $K(\tau)=\frac{1}{\tau^{2}}J(-\frac{1}{\tau})-J(\tau)$, we have
     $K(\tau+2)-K(\tau)\simeq 0$ for $|\tau|>>0$ and
     $\frac{1}{(i\tau)^{2}}K\left(-\frac{1}{\tau}\right)-K(\tau)= 0$.
     This is because we have
     $\frac{1}{\tau^{2}}G_{2}(-\frac{1}{\tau})=G_{2}(\tau)- \frac{\pi
     i}{\tau}$ and $G_{2}(\tau+1)=G_{2}(\tau)$ \cite{Zag08}.

     The quasimodular form $J(\tau)$ is essentially obtained in the
     discussion of \cite[Section 3.2]{Oka09} for a spherical object. The
     point of Theorem \ref{thm:quasimodular} is to attach a quasimodular
     form to quintic periods, using the monodromy around the Gepner
     point.

      \section{The GKZ system of the $A_{1}$ singularity}\label{sec:a1}
      Let us discuss the GKZ system of $A_{1}$ singularity
      $\bC^{2}/\bZ_{2}$.  For $a_{i}\in \bC$,
      $x=\frac{a_{1}a_{3}}{a_{2}^2}$, polynomials $a_{1}+a_2 W+a_{3}
      W^2$, and its roots $\beta_{0}(x)$ and $\beta_{1}(x)$, let us put
      the following:
      \begin{align*}
       \varpi_{0}(x)&=1, \\
       \varpi_{1}(x)&=\frac{\log\beta_{0}(x)-\log \beta_{1}(x)}{2\pi i}.
       \end{align*}
       
       By \cite[Proposition 4.4]{Hos04}, $\varpi_{0}(x)$ and
       $\varpi_{1}(x)$ give a basis of solutions of the GKZ system of
       the $A_{1}$ singularity, and near the large complex structure
       limit $x=0$, $\varpi_{1}(x)$ can be written as a hypergeometric
       series.  By \cite[Section 3]{Hos04}, $\varpi_{0}(x)$ and
       $\varpi_{1}(x)$ give a basis of central charges of the GKZ system
       of the $A_{1}$ singularity.  For simplicity, we say $A_{1}$
       periods for solutions of the GKZ system of the $A_{1}$
       singularity.

       The GKZ system of the $A_{1}$ singularity and its solutions are
       closely related to K. Saito's differential equations and
       primitive forms of the $A_{1}$ singularity \cite[Propositions 4.1
       and 4.2]{Hos04}.
       \subsection{Doubled Kronecker-quiver case}\label{sec:dK}
       For the cotangent bundle of $\bP^{1}$, denoted by $T^{*}\bP^{1}$,
       let $\Db_{\bP^{1}}(\Coh T^{*}\bP^1)$ be the full subcategory of
       $\Db(\Coh T^{*}\bP^1)$ consisting of objects supported over
       $\bP^{1}\stackrel{i}{\hookrightarrow} T^{*}\bP^1$.  Let us
       mention that $T^{*}\bP^{1}$ can be obtained by the minimal
       resolution of the $A_{1}$ singularity.

       Let us recall the doubled Kronecker quiver $K$, which has two
       vertices, two parallel arrows $a,b$, and two more inverse arrows
       $a',b'$ with commuting relations $b'a=a'b$ and $ba'=ab'$.  We
       have the following figure:
       \begin{figure}[H]
	\begin{center}
	 \def\svgwidth{90pt}
	 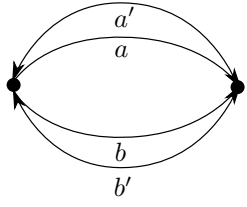
	 \caption{Doubled Kronecker quiver (with the commuting relations)}
	\end{center}
       \end{figure}

       Let $\mo^{nil} K$ be the category of nilpotent representations of
       the doubled Kronecker quiver, which is a heart of
       $\Db_{\bP^{1}}(\Coh T^{*}\bP^1)$ \cite{CraHol}.

       We have Bridgeland stability conditions on $\mo^{nil}K$ for any
       non-zero central charges of simple representations, which are
       isomorphic to objects $i_{*}\cO_{\bP^{1}}(-1)$ and
       $i_{*}\cO_{\bP^{1}}(-2)[1]$.

        For a point $y\in \bP^{1}$,
	let central charges of $A_{1}$ periods, denoted by $Z_{x}^{A_{1}}$, 
	be defined as follows \cite[Section 3]{Hos04}:
       \begin{align*}
	Z_{x}^{A_{1}}(i_{*}\cO_{y})=\varpi_{0}(x),\\
	Z_{x}^{A_{1}}(i_{*}\cO_{\bP^{1}}(-1))=\varpi_{1}(x).
       \end{align*}
       Near the large complex structure limit, $\varpi_{1}(x)$ is
       asymptotically $\frac{\log(x)}{2\pi i}$.

       \begin{remark}\label{rmk:A1}
	For $A_{1}$ periods, we can just take Bridgeland stability
	conditions, since for objects of the heart $\mo^{nil} K$, we do
	not have $A_{1}$ periods which are asymptotically some several
	powers of $\frac{\log(x)}{2\pi i}$.
       \end{remark}
       
	As an analog of Corollary \ref{cor:analytic_deform}, we have the
	following proposition.

	\begin{proposition}\label{prop:analytic_dk}
	 For Bridgeland stability conditions on the heart $\mo^{nil} K$
	 of $\Db_{\bP^1}(T^{*}\bP^1)$ of central charges $Z_{x}^{A_{1}}$
	 of $A_{1}$ periods, $A_{1}$ periods give wall-crossings of the
	 second kind, which are different from ones of dilation and
	 rotation.
	\end{proposition}
	\begin{proof}
	 The central charge of the object $i_{*}\cO_{y}$ for a point $y$
	 is already fixed.  So, the assertion follows by Lemma
	 \ref{lem:deform}.
	\end{proof}

       	By HMS \cite{IUU} (in particular the proof of \cite[Theorem
       	28]{IUU}), in an affine manifold in the formulation of Fukaya
       	categories of \cite{IUU}, simple representations correspond to
       	Lagrangian spheres.
	
	The mirror map is defined as $e^{2\pi i \varpi_{1}(x)}$.  As an
	analog of Corollary \ref{cor:mirror_map}, we have the following
	proposition.

       \begin{proposition}\label{prop:a1_mirror_map}
	For Bridgeland stability conditions on the heart $\mo^{nil} K$
	of $\Db_{\bP^1}(T^{*}\bP^1)$ of central charges $Z_{x}^{A_{1}}$
	of $A_{1}$ periods, by dilation and rotation with the $A_{1}$
	period $\varpi_{1}(x)$, the sum of central charges of Lagrangian
	simple objects gives the mirror map.
       \end{proposition}
       \begin{proof}
	We have $\varpi_{1}(x)Z_{x}^{A_{1}}(i_{*}\cO_{y})= \varpi_{1}(x)$
	and take the exponential function.
       \end{proof}
       
	As an analog of Corollary \ref{cor:PF}, we have the following
	proposition.

	\begin{proposition}\label{prop:a1_PF}
	 For Bridgeland stability conditions of
	 $\Db_{\bP^1}(T^{*}\bP^1)$, we have an one-parameter subspace
	 such that central charges of Lagrangian simple objects of the
	 heart $\mo^{nil} K$ give a basis of $A_{1}$ periods.
	\end{proposition}
	\begin{proof}
	 For central charges of $Z_{x}^{A_{1}}$ of $A_{1}$ periods, we
	 have $\varpi_{1}(x)$ and $1-\varpi_{1}(x)$ for simple
	 representations of $\mo^{nil} K$.
	\end{proof}

	We have autoequivalences on $\Db_{\bP^{1}}(\Coh T^{*}\bP^1)$
	\cite{IshUeh} which give monodromy actions of the GKZ system
	\cite{dFS}.  In particular, we have the automorphism of the
	doubled Kronecker quiver, which we denote by $M$.  Up to
	monodromy actions, $A_{1}$ periods are given by central charges
	of $A_{1}$ periods.

	\begin{remark}\label{rmk:a_1-wall}
	 Unlike for the quiver $A_{3}^{\otimes 5}$ and central charges
	 of quintic periods, the automorphism $M$ for the doubled
	 Kronecker quiver and central charges of $A_{1}$ periods is
	 non-trivial.  In particular, as the disappearance of the stable
	 object $i_{*}\cO_{y}$ or $M(i_{*}\cO_{y})$ for a point $y\in
	 \bP^1$, we have the wall-crossing of first kind.
	\end{remark}

    \subsection{$A_{1}$-quiver case}\label{sec:a1_quiver}

     For the function $W^{3}:\bC\to \bC$, denoted by $H$, we have the
     derived Fukaya-Seidel category $\FS(H)$, which is given by a
     morsification of $H$ such as $a_{1}+a_{2}W +a_{3} W^{2}+W^3- V^2$
     for $a_{i}\in \bC$.  A heart of $\FS(H)$ is isomorphic to $\mo
     A_{1}$.

     As in Section \ref{sec:HMS}, we have a morsification of $H$ by
     A'Campo, which is of elliptic curves.  For such a morsification of
     $H$, which we keep taking in the following, simple representations
     of $\mo A_{1}$ correspond to Lagrangian vanishing cycles in the
     zero locus of the morsification in the formulation of Fukaya-Seidel
     categories.

     On $\mo A_{1}$, we have Bridgeland stability conditions for any
     non-zero central charges of simple representations as in Section
     \ref{sec:dK}. Central charges of $A_{1}$ periods can be regarded as
     central charges of $\mo A_{1}$, by embedding the quiver $A_{1}$ into
     the doubled Kronecker quiver $K$.  For Bridgeland stability
     conditions on $\mo A_{1}$, let us call simple representations as
     Lagrangian stable objects.

     Let us recall that from \cite[Example 2.9]{Orl09}, for the ring
     $A=\frac{\bC[x]}{x^{3}}$, we have the graded category of the
     singularity $D_{Sg}^{gr}(A)$, which is an algebro-geometric derived
     category, such that we have
     \begin{align*}
      D_{Sg}^{gr}(A)\cong \Db(\mo A_{1})\cong \FS(H).
     \end{align*}

     For analogs of Corollaries \ref{cor:analytic_deform},
     \ref{cor:mirror_map}, and \ref{cor:PF}, in Propositions
     \ref{prop:analytic_dk}, \ref{prop:a1_mirror_map}, and
     \ref{prop:a1_PF}, we simply replace $\mo^{nil}K$ and
     $\Db_{\bP^1}(\Coh T^{*}\bP^1)$ with $\mo A_{1}$ and
     $D_{Sg}^{gr}(A)$.  Let $t$ be the non-simple irreducible
     representation of $\mo A_{1}$.  For their proofs, in $K(\mo^{nil}
     K)$, the class of the embedded representation of the representation
     $t$ is that of the object $i_{*}\cO_{y}$ for a point $y\in
     \bP^{1}$. So, we simply replace the object $i_{*}\cO_{y}$ by the
     representation $t$ in the proofs of Propositions
     \ref{prop:analytic_dk}, \ref{prop:a1_mirror_map}, and
     \ref{prop:a1_PF}.

     We have analogs of Theorems \ref{thm:sheaf} and \ref{thm:lag} for
     the quiver $A_{1}$ by taking central charges of $A_{1}$ periods for
     an embedding of the $A_{1}$ quiver into the doubled Kronecker
     quiver $K$.  Also, as in Section \ref{sec:dK}, we have the
     wall-crossing of first kind by $A_{1}$ periods.

     \section{Acknowledgments}
     The author thanks Professors S. Galkin, K. Gomi, D. Honda,
     S. Hosono, M. Kashiwara, M. Kontsevich, H. Nakajima, R. Ohkawa,
     K. Ueda, and S. Yanagida for their helpful discussions.

 \bibliographystyle{amsalpha}
 
 \end{document}